\documentclass[letterpaper,11pt,twoside,keywordsasfootnote,addressatend,noinfoline]{article}

\usepackage{imsart}
\usepackage{fullpage}
\usepackage[english]{babel}
\usepackage{amssymb}
\usepackage{amsmath}
\usepackage{amsthm}
\usepackage{mathrsfs}
\usepackage{hyperref}
\hypersetup{hidelinks}
\usepackage{enumitem}

\newtheorem{theorem}{Theorem}
\renewenvironment{proof}{\noindent{\bf Proof.}}{\hspace*{2mm}~$\square$}

\newcommand{\R}{\mathbb{R}}
\newcommand{\B}{\mathscr{B}}
\newcommand{\ind}{\mathbf{1}}
\newcommand{\n}{\hspace*{-5pt}}

\DeclareMathOperator{\poisson}{Poisson}
\DeclareMathOperator{\uniform}{Uniform}
\DeclareMathOperator{\binomial}{Binomial}
\DeclareMathOperator{\card}{card}

\begin{document}

\begin{frontmatter}
\title{Short proof of the conditioning property for \\ multi-dimensional Poisson point processes}
\runtitle{Conditioning property for multi-dimensional Poisson point processes}
\author{Nicolas Lanchier}
\runauthor{Nicolas Lanchier}
\address{School of Mathematical and Statistical Sciences \\ Arizona State University \\ Tempe, AZ 85287, USA. \\ nicolas.lanchier@asu.edu}
\maketitle

\begin{abstract} \ \
 Poisson processes and one-dimensional Poisson point processes satisfy three main properties: superposition, thinning, and conditioning.
 The proof of the first two relies on basic estimates involving the Poisson distribution that are also true for multi-dimensional Poisson point processes.
 In contrast, the proof of conditioning uses that the distances between consecutive occurrences in time or entities in space are independent and exponentially distributed, which is nonsensical in higher dimensions.
 This paper gives a short proof of the conditioning property for multi-dimensional Poisson point processes.
\end{abstract}

\begin{keyword}[class=AMS]
\kwd[Primary ]{60G55}
\end{keyword}

\begin{keyword}
\kwd{Poisson point processes; Conditioning property; Uniform random variable.}
\end{keyword}

\end{frontmatter}

%%%%%%%%%%%%%%%%%%%%%%%%%%%%%%%%%%%%%%%%%%%%%%%%%%%%%%%%%%%%%%%%%%%%%%%%%%%%%%%%%%%%%%%%%%%%%%%%%%%%%%%%%%%%%%%%%%%%%%%%%%%%%%%%%%%%%%%%%%%%%%%%%%%%%%%%%%%%%%%%%%%%%%%%%%%%%%%%%%%%%%%%%%%%%%%%%%

\section{Poisson and Poisson point processes}
 The Poisson random variable can be viewed as the limit of binomial random variables as the number of trials~$n$ goes to infinity and the success probability scales like~$1/n$, making this random variable the continuous-time analog of the binomial random variable.
 Poisson processes and Poisson point processes are derived from the Poisson random variable, and are used respectively to model the distribution of occurrences in time and the distribution of entities in space.
 In particular, Poisson point processes are extensions of Poisson processes to potentially higher dimensions.
 More precisely, the process~$X = \{X_t : t \in [0, \infty) \}$ is the Poisson process with rate~$\mu$ if it starts at~$X_0 = 0$ and if it has independent Poisson increments, meaning that
\begin{itemize}[leftmargin=36pt]
\item[(A1)] For all~$t_1 < t_2 < \cdots < t_{2k - 1} < t_{2k}$, $X_{t_2} - X_{t_1}, \ldots, X_{t_{2k}} - X_{t_{2k - 1}}$ are independent. \vspace*{4pt}
\item[(A2)] For all~$s < t$, $X_t - X_s = \poisson (\mu (t - s))$.
\end{itemize}
 The random variable~$X_t$ is interpreted as the number of occurrences until time~$t$.
 This concept can be extended to higher dimensions using a counting process
 $$ N = \{N (B) : B \in \B (\R^d) \ \hbox{bounded} \} \quad \hbox{where} \quad \B (\R^d) = \hbox{Borel~$\sigma$-algebra} $$
 and where~$N (B)$ is interpreted as the number of entities in the Borel set~$B$.
 The process~$N$ is the Poisson point process with intensity~$\mu$ if the number of entities, called Poisson points, in nonoverlapping subsets are independent, and the number of entities in a given subset is Poisson distributed.
 More precisely, letting~$\lambda$ refer to the Lebesgue measure,
\begin{itemize}[leftmargin=40pt]
\item[(A1')] For all~$B_1, B_2, \ldots, B_k \in \B (\R^d)$ disjoint, $N (B_1), N (B_2), \ldots, N (B_k)$ are independent. \vspace*{4pt}
\item[(A2')] For all~$B \in \B (\R^d)$ bounded, $N (B) = \poisson (\mu \lambda (B))$.
\end{itemize}
 Note that the Poisson point process~$N$ on the real line~(for~$d = 1$) and the Poisson process~$X$ are mathematically equivalent in the sense that
 $$ X_t = N ([0, t]) \quad \hbox{in distribution \ for all \ $t \geq 0$}. $$
 These two processes satisfy three main properties, called superposition, thinning, and conditioning, that can be described for the one-dimensional Poisson point process~$N$ as follows. \vspace*{4pt} \\
\noindent{\bf Superposition}.
 Having a palette of~$n$ colors, labeled~$i = 1, 2, \ldots, n$, assume that entities with color~$i$ are randomly distributed on the real line according to independent Poisson point processes with intensity~$\mu_i$.
 Then, the set of entities of any color~(what you see if you become colorblind) is distributed according to a Poisson point process with intensity~$\mu_1 + \mu_2 + \cdots + \mu_n$. \vspace*{4pt} \\
\noindent{\bf Thinning}.
 Having a palette of~$n$ colors, labeled~$i = 1, 2, \ldots, n$, and entities randomly distributed on the real line according to a Poisson point process with intensity~$\mu$, paint each entity independently color~$i$ with probability~$p_i$ where~$p_1 + p_2 + \cdots + p_n = 1$.
 Then, the~$n$ sets of entities with different colors are distributed according to independent Poisson point processes with intensity~$\mu p_i$. \vspace*{4pt} \\
\noindent{\bf Conditioning}.
 Having entities randomly distributed on the real line according to a Poisson point process with intensity~$\mu$, given that there are~$n$ entities in~$[0, t]$, the position of these entities is described by~$n$ independent uniform random variables on the interval~$[0, t]$. \vspace*{4pt} \\
 From the point of view of the Poisson process~$X$, this can be formulated as
\begin{theorem}[Conditioning property]
\label{th:poisson}
 Given that~$X_t = n$,
\begin{itemize}
\item let $U_1, U_2, \ldots, U_n$ be independent~$\uniform (0, t)$, \vspace*{4pt}
\item let $\tau_i = \inf \,\{s : X_s = i \}$ for~$i = 1, 2, \ldots, n$ be the times of the occurrences in~$(0, t)$.
\end{itemize}
 Then~$\{U_1, U_2, \ldots, U_n \} = \{\tau_1, \tau_2, \ldots, \tau_n \}$ in distribution.
\end{theorem}
\noindent
 Superposition and thinning can be viewed as converse of each other.
 Their proof relies on basic estimates involving the Poisson random variable that are not sensitive to the dimension, so they are also true for multi-dimensional Poisson point processes.
 In contrast, the main ingredient to prove the conditioning property is that the distances between consecutive occurrences in time or entities in space are independent and exponentially distributed with parameter~$\mu$, a property that is nonsensical in higher dimensions~(see~\cite[Lemma~9.4]{lanchier_2017} or~\cite[Theorem~5.2]{ross_2010} for a proof).
 Also, though the conditioning property is expected to hold for multi-dimensional Poisson point processes, the traditional proof in one dimension does not extend to higher dimensions, and the author could not find any proof in the probability literature.
 The main objective of this paper is to give a short proof of the conditioning property for multi-dimensional Poisson point processes.
 
%%%%%%%%%%%%%%%%%%%%%%%%%%%%%%%%%%%%%%%%%%%%%%%%%%%%%%%%%%%%%%%%%%%%%%%%%%%%%%%%%%%%%%%%%%%%%%%%%%%%%%%%%%%%%%%%%%%%%%%%%%%%%%%%%%%%%%%%%%%%%%%%%%%%%%%%%%%%%%%%%%%%%%%%%%%%%%%%%%%%%%%%%%%%%%%%%%

\section{The Conditioning property}
 To extend Theorem~\ref{th:poisson} to higher dimensions, we let~$N$ be the~$d$-dimensional Poisson point process with intensity~$\mu$, and we let~$B$ be a bounded Borel subset of~$\R^d$.
 The uniform random variable on this Borel set is the continuous random variable with density function
 $$ \phi : \R^d \to \R_+ \quad \hbox{defined as} \quad \phi = \ind_B / \lambda (B). $$
 Then, we have the following result that the author also explains in the video~\cite{lanchier_2024}.
\begin{theorem}[Conditioning property]
 Given that~$N (B) = n$,
\begin{itemize}
\item let $U_1, U_2, \ldots, U_n$ be independent~$\uniform (B)$, \vspace*{4pt}
\item let $P_1, P_2, \ldots, P_n \in \R^d$ be the Poisson points in~$B$.
\end{itemize}
 Then~$\{U_1, U_2, \ldots, U_n \} = \{P_1, P_2, \ldots, P_n \}$ in distribution in the sense that
 $$ P (\card (A \cap \{U_1, U_2, \ldots, U_n \}) = k) = P (\card (A \cap \{P_1, P_2, \ldots, P_n \}) = k) $$
 for all Borel subsets~$A \subset B$ and all~$k = 0, 1, \ldots, n$.
\end{theorem}
\begin{proof}
 Let~$a = \lambda (A)$ and~$b = \lambda (B)$.
 By definition of the uniform random variable,
 $$ \begin{array}{c} P (U_i \in A) = \int \phi \,\ind_A \,d\lambda = \lambda (A) / \lambda (B) = a/b \quad \hbox{for all} \quad i = 1, 2, \ldots, n. \end{array} $$
 In particular, using also independence, for all~$k = 0, 1, \ldots, n$,
\begin{equation}
\label{eq:uniform}
\begin{array}{rcl}
  P (\card (A \cap \{U_1, U_2, \ldots, U_n \}) = k) & \n = \n & P (\card \{i = 1, 2, \ldots, n : U_i \in A \} = k) \vspace*{4pt} \\
                                                    & \n = \n & P (\binomial (n, a/b) = k). \end{array}
\end{equation}
 Now, because~$A \cap (B \setminus A) = \varnothing$, it follows from axiom~(A1') that
\begin{equation}
\label{eq:poisson1}
\begin{array}{l}
  P (\card (A \cap \{P_1, P_2, \ldots, P_n \}) = k \,| \,N (B) = n) \vspace*{4pt} \\ \hspace*{20pt} =
  P (\card \{i = 1, 2, \ldots, n : P_i \in A \} = k \,| \,N (B) = n) \vspace*{4pt} \\ \hspace*{20pt} =
  P (N (A) = k \,| \,N (B) = n) \vspace*{4pt} \\ \hspace*{20pt} =
  P (N (A) = k, N (B) = n) / P (N (B) = n) \vspace*{4pt} \\ \hspace*{20pt} =
  P (N (A) = k, N (B \setminus A) = n - k) / P (N (B) = n) \vspace*{4pt} \\ \hspace*{20pt} =
  P (N (A) = k) P (N (B \setminus A) = n - k) / P (N (B) = n). \end{array}
\end{equation}
 Using also that~$\lambda (B \setminus A) = \lambda (B) - \lambda (A) = b - a$ and axiom~(A2') of the definition of a Poisson point process, then simplifying, the last term in~\eqref{eq:poisson1} becomes
\begin{equation}
\label{eq:poisson2}
\begin{array}{l}
% \displaystyle \frac{P (\poisson (\mu a) = k) \,P (\poisson (\mu (b - a)) = n - k)}{P (\poisson (\mu b) = n)} \vspace*{12pt} \\ \hspace*{40pt} =
\displaystyle \bigg(\frac{\mu^k a^k \,e^{- \mu a}}{k!} \bigg) \bigg(\frac{\mu^{n - k} (b - a)^{n - k} \,e^{- \mu (b - a)}}{(n - k)!} \bigg) \bigg/ \bigg(\frac{\mu^n b^n \,e^{- \mu b}}{n!} \bigg) \vspace*{12pt} \\ \hspace*{40pt} =
\displaystyle \bigg(\frac{a^k}{k!} \bigg) \bigg(\frac{(b - a)^{n - k}}{(n - k)!} \bigg) \bigg/ \bigg(\frac{b^n}{n!} \bigg) =
\displaystyle \frac{n!}{(n - k)! \,k!} \ \bigg(\frac{a^k (b - a)^{n - k}}{b^k \,b^{n - k}} \bigg) \vspace*{12pt} \\ \hspace*{40pt} =
\displaystyle {n \choose k} \bigg(\frac{a}{b} \bigg)^k \bigg(1 - \frac{a}{b} \bigg)^{n - k} =
\displaystyle P (\binomial (n, a/b) = k). \end{array}
\end{equation}
 The result follows noticing that the last terms in~\eqref{eq:uniform} and~\eqref{eq:poisson2} are equal.
\end{proof}

%%%%%%%%%%%%%%%%%%%%%%%%%%%%%%%%%%%%%%%%%%%%%%%%%%%%%%%%%%%%%%%%%%%%%%%%%%%%%%%%%%%%%%%%%%%%%%%%%%%%%%%%%%%%%%%%%%%%%%%%%%%%%%%%%%%%%%%%%%%%%%%%%%%%%%%%%%%%%%%%%%%%%%%%%%%%%%%%%%%%%%%%%%%%%%%%%%

\end{document}